\documentclass[a4paper,11pt]{amsart}
\usepackage{amsmath,amssymb,colordvi}
\usepackage{verbatim}
\newtheorem{theorem}{Theorem}
\newtheorem{corollary}[theorem]{Corollary}
\newtheorem{lemma}[theorem]{Lemma}

\newtheorem{proposition}[theorem]{Proposition}
\newtheorem{definition}[theorem]{Definition}

\newtheorem{remark}[theorem]{\it Remark}

\usepackage{amsthm}

\usepackage{mathrsfs}

\usepackage{amsmath}

\usepackage{color}

\usepackage{amssymb}

\def\NN{\mathbb N}
\def\N{\mathbb N}

\def\R{\mathbb R}
\def\C{\mathbb C}
\def\RR{\mathbb R}

\begin{document}

\title[Hermite expansions of $C_0$-groups and cosine functions]
{Hermite expansions of $C_0$-groups and cosine functions}

\author{Luciano Abadias}
\address{Departamento de Matem\'aticas, Instituto Universitario de Matem\'aticas y Aplicaciones, Universidad de Zaragoza, 50009 Zaragoza, Spain.}

\email{labadias@unizar.es}

\author{Pedro J. Miana}
\address{Departamento de Matem\'aticas, Instituto Universitario de Matem\'aticas y Aplicaciones, Universidad de Zaragoza, 50009 Zaragoza, Spain.}
\email{pjmiana@unizar.es}

\thanks{Authors have been partially supported by Project MTM2010-16679, DGI-FEDER, of the MCYTS and Project E-64, D.G. Arag\'on,Spain.}




\begin{abstract}
In this paper we introduce vector-valued Hermite expansions to approximate  one-parameter operator families  such as $C_0$-groups and cosine functions. In both cases we  estimate  the rate of convergence of these Hermite expansions to the related family and compare with other known approximations. Finally we illustrate our results with particular examples of $C_0$-groups and cosine functions and their Hermite expansions.
\end{abstract}

\date{}

\maketitle

\section{Introduction}

Representations of functions through expansions of orthogonal polynomials such as Legendre, Hermite or Laguerre are well known in the classical analysis. They allow to approximate functions by series of orthogonal polynomials on different types of convergence: a pointwise way, uniformly in  certain sets, or in Lebesgue norm. Two classical monographs where we can find this kind of results are \cite[Chapter 4]{Lebedev} and \cite[Chapter IX]{Szego}. In this paper we are concentrated on Hermite expansions.

For all $n\in\N\cup \{ 0 \},$ classical Hermite polynomials are defined by Rodrigues' formula $$H_n(t):=(-1)^n e^{t^2}\frac{d^n}{dt^n}(e^{-t^2})(t),\qquad t\in\R.$$  We mention an interesting theorem which may be found in \mbox{\cite[Theorem 2, Sec. 4.15]{Lebedev}} and whose  original statement was given in \cite{Uspensky}. Although this result holds for functions which are more general that stated in the next theorem (see, for example \cite[p. 603]{Lebedev} and \cite[Theorem 9.1.6]{Szego}) the following version is sufficient for our purposes.

\begin{theorem}\label{ExpEsc} Let $f:\R\to\C$ be a differentiable function such that $\displaystyle{\int_{-\infty}^{\infty}e^{-t^2}|f(t)|^2\,dt<+\infty}$. Then the series $\displaystyle\sum_{n=0}^{\infty}c_n(f)\,H_n(t),$ where $$c_n(f):=\frac{1}{2^n n!\sqrt{\pi}}\int_{-\infty}^{\infty}e^{-t^2}f(t)H_n(t)\,dt,$$ converges pointwise to $f(t)$ for $t\in \RR$.
\end{theorem}

For $\lambda \in \C$, we consider the exponential function $e_{\lambda}$ where $e_{\lambda}(t):=e^{\lambda t}$ for $t\in \R$. Observe that  the function $e_{\lambda}$ is expressed in a Hermite expansion by \begin{equation}\label{expo}e^{\lambda t}=\displaystyle\sum_{n=0}^{\infty}\frac{1}{2^n n!} \lambda^ne^{\frac{\lambda^2}{4}}H_n(t),\qquad t\in \R,\end{equation} for $\lambda\in \C$, see \cite[Example 2, p.74]{Lebedev}.

Now let $c(t):=\cos(\sqrt{a}\,t)$ with $a>0$ for $t\in\R.$
This function also satisfies the hypothesis of Theorem \ref{ExpEsc}, and we write
\begin{eqnarray}\label{coseno}
\cos(\sqrt{a}\,t)= \sum_{n=0}^{\infty}\frac{1}{2^{2n}(2n)!}(-a)^ne^{-\frac{a}{4}}H_{2n}(t), \qquad t\in \R.
\end{eqnarray}

Different aspects of Hermite expansions have been studied in the literature: for example, the asymptotic behavior of coefficients in \cite{Boyd};
expansions of analytic functions in  \cite{Rusev}; and estimations of coefficients and integrals related with Hermite expansions in \cite{Muckenhoupt1}.

Other approximation results  involve Hermite functions defined by $$\mathcal{H}_n(t):=\frac{1}{(2^n n!{\pi^{\frac{1}{2}}})^{1\over 2}}e^{-\frac{t^2}{2}}H_n(t),\quad t\in\R.$$ These functions form an orthonormal basis of the Hilbert space $L^2(\R).$ Moreover, let $f$ be in $L^p(\R),$ $\frac{4}{3}<p<4,$ and $a_k(f):=\int_{-\infty}^{\infty}f(t)\mathcal{H}_k(t)\,dt$ for $k\in \N\cup\{0\}$. Then $\lVert S_n(f)-f \rVert_p\to 0$ as $n\to\infty,$ with $$S_n(f):=\displaystyle\sum_{k=0}^{n}a_k(f)\mathcal{H}_k,\qquad n\in\N\cup\{0\},$$ see \cite[Theorem 2]{Askey}.

Also these Hermite series allow to approximate temperated distributions. Let $\mathcal S$ be the Schwartz class and $\mathcal S'$ the space of temperated distribution on $\R$. Then the  equality $$
T=\sum_{k=0}^\infty \langle T, \mathcal{H}_k\rangle \mathcal{H}_k
$$
holds in the weak sense, see for example \cite[pp. 143]{Reed}. Moreover,  Hermite expansions of Dirac distribution and the distribution principal value of ${1\over x}$ may be found in \cite[pp 191-193]{AMS} and \cite[Section 2]{CMO}. Hermite expansions of  products  of temperated distributions are considered in detail in \cite{CMO}.

In other hand, a $C_0$-group $(T(t))_{t\in\R}$ is a one parameter (strongly continuous) family of linear and bounded operators on a Banach space $X,$ which satisfies the exponential law, i.e., $T(t+s)=T(t)T(s)$ for all $t,s\in\R,$ and $T(0)=I.$ The (densely defined) operator $A$ given by $$Ax:=\displaystyle\lim_{t\to 0}\frac{T(t)x-x}{t},$$ when the limit exists ($x\in D(A)$), is called the infinitesimal generator of $C_0$-group; moreover, for $t\in\R$, $$\frac{d}{dt}T(t)x=T(t)Ax, \qquad x\in D(A),$$
see more details in monographs \cite{ABHN, Nagel}. This family of operators is interpreted as $(e^{tA})_{t\in\R},$ since it is the solution of the abstract Cauchy problem of first order. Likewise, it seems natural to consider identity (\ref{expo}) in the vector-valued version, and we have that $$T(t)x=\displaystyle\sum_{n=0}^{\infty}\frac{1}{2^{n} n!}A^nT^{(g)}\biggl(\frac{1}{4}\biggr)x\,H_{n}(t), \qquad t\in \R,$$ for $x\in D(A)$, Theorem \ref{loc} (ii), where $T^{(g)}(z)$ is a holomorphic $C_0$-semigroup of angle $\frac{\pi}{2}$ whose generator is $A^2.$  It would be nice to obtain the larger domain (might be the domain of fractional powers of $-A$, $D((-A)^\alpha)$, with $0<\alpha<1$) where the Hermite expansion converges.

We also consider the second order abstract Cauchy problem, whose solutions may be expressed in terms of cosine functions. A cosine function $(C(t))_{t\in \R}$ is a one-parameter (strongly continuous) family of linear and bounded operators on a Banach space $X,$ which satisfies $C(t+s)+C(t-s)=2C(t)C(s)$ for $s,t\in \R,$ and $C(0)=I.$ The (densely defined) operator $A$ given by $$Ax:=\displaystyle\lim_{t\to 0}\frac{2}{t^2}(C(t)x-x),$$ for $x\in D(A),$ is the generator of the cosine function; moreover, for $t\in\R$, $$\frac{d^2}{dt^2}C(t)x=C(t)Ax, \qquad x\in D(A),$$
see other details in \cite[Section 3.14]{ABHN}. In Theorem \ref{loc2} (ii), we show that
$$
C(t)x= \displaystyle\sum_{n=0}^{\infty}\frac{1}{2^{2n}(2n)!}A^n T^{(c)}\biggl(\frac{1}{4}\biggr)x\,H_{2n}(t), \qquad t\in\R,
$$
for $x\in D(A);$ in this case $A$ also generates a holomorphic $C_0$-semigroup of angle $\frac{\pi}{2}$, $(T^{(c)}(z))_{z\in \C^+}$. Note that this Hermite expansion extends the formula (\ref{coseno}).

In the literature, there exist many different approximations  of $C_0$-semigroups, as Euler, Yosida, Dunford-Segal or subdiagonal Pad\'{e} approximations, see \cite{Gomilko} and references in \cite{Abadias}. However  there are not some many  approximations of $C_0$-groups and cosine functions: stable rational approximations for exponential function are considered to treat hyperbolic problems, i.e., $C_0$-groups in  \cite{BT} and cosine functions in \cite[Section 4]{jara}. In \cite[Theorem 4.2]{Ba}, the author approximates the solution of fractional Cauchy problem of order $\alpha \ge 1$ and gives the rate of convergence;  the case of cosine functions is included for $\alpha=2$.

The paper is organized as follows. In the second section, we consider the functions $t\mapsto\frac{1}{2^n n!\sqrt{\pi}}e^{-t^2}H_n(t)$ (for $n\in \NN\cup\{0\}$) which have a key role in Theorem \ref{ExpEsc}. They satisfy interesting properties, similar to Hermite polynomials (Proposition \ref{propHerm}). We estimate their $p$-norm in Theorem \ref{main2} and obtain Hermite expansions for Dirichlet and Fej\'{e}r  kernels in Proposition \ref{kernel}. We also consider Hermite expansions in Lebesgue space $L^p(\RR)$ for $1\le p<\infty$.  Finally, we give a vector-valued version of Theorem \ref{ExpEsc} on an abstract Banach space $X$.

Main theorems of this paper appear in the third and forth  sections. In Theorem \ref{loc} and Theorem \ref{loc2}, we express $C_0$-groups and cosine functions though Hermite expansions. In Theorem \ref{rate1} and Theorem \ref{rate2}, we give the rate of the Hermite expansions to the $C_0$-group and cosine function respectively, which depends on the regularity of the initial data. We compare with the  Laguerre expansions  obtained for the case of $C_0$-semigroups in \cite[Theorem 5.2]{Abadias}. Two consequences of these Hermite expansions for $C_0$-groups and cosine functions are the Hermite expansions for F\'{e}jer operators (Corollary \ref{corola} and Theorem \ref{rate2} (iv)) and  series representations for subordinated holomorphic semigroups (Theorem \ref{loc} (iii) and \ref{loc2}(iii)).

 Although, cosine functions and $C_0$-semigroups differ really (see for example \cite{Bobrowski}), the nature of $C_0$-groups and cosine functions are quite similar. In the case of Hermite expansions, analogous results hold (compare Theorem \ref{loc} and \ref{loc2}) and both approaches are compatible, see Remark \ref{conn} and subsection \ref{sub}.

In the last section, we present some examples of $C_0$-groups and cosine functions and their Hermite expansions: shift and multiplication groups, cosine functions on sequence spaces, and matrix approach to cosine functions and $C_0$-groups. We also give some comments and open questions to motivate forthcoming papers in this topic.





\noindent {\bf Notation}. Given $1\leq p <\infty,$ let $L^p(\R)$ be the set of Lebesgue $p$-integrable functions, that is, $f$ is a measurable function and $$\Vert f \Vert_p:=\biggl(\int_{-\infty}^{\infty}|f(t)|^p\,dt\biggr)^{\frac{1}{p}}<\infty;$$ for $p=2$, remind that $L^2(\R)$ is a Hilbert space with $\langle\,\,,\,\, \rangle$ the usual inner product, and $L^{\infty}(\R)$ the set of essential bounded Lebesgue functions with the norm $\lVert f\rVert_{\infty}:=\displaystyle\hbox{esssup}_{t\in\R}|f(t)|.$ We call $C_0(\R)$ the set of continuous functions defined in $(-\infty,\infty)$ such that $\displaystyle\lim_{t\to \pm\infty}f(t)=0,$ with the norm $\lVert \ \rVert_{\infty}.$

\section{Hermite functions and Hermite expansions on Banach spaces}

\setcounter{theorem}{0}
\setcounter{equation}{0}

Hermite polynomials are solutions of second order differential equation \begin{equation}\label{ecdif}y''-2ty'+2ny=0, \qquad n\ge 0, \quad t\in \R.\end{equation} Furthermore, they satisfy the following condition of orthogonality: $$\frac{1}{2^n n!\sqrt{\pi}}\int_{-\infty}^{\infty}H_n(t)H_m(t)e^{-t^2}\,dt=\delta_{n,m},\qquad n,m\in \N\cup\{0\},$$ where $\delta_{n,m}$ is the Kronecker delta. These polynomials satisfy the following properties:
\begin{equation}\label{pro1}H_n(-t)=(-1)^nH_n(t),\ n\in \N\cup\{0\},\end{equation}
\begin{equation*}\label{pro2}H_{2m}(0)=(-1)^m\frac{(2m)!}{m!}\text{ and }H_{2m+1}(0)=0,\end{equation*}
\begin{equation}\label{pro4}H_{n+1}(t)=2tH_{n}(t)-2nH_{n-1}(t);\text{  }H_0(t)=1;\text{  }H_1(t)=2t;\end{equation}
for $t\in \R$.
\bigskip

From here on out, Muckenhoupt estimates for $\mathcal{H}_n$ and $H_n$, which are well known in the classical theory of orthogonal polynomials, will be used several times: there exist  constants $\gamma>0$ and $C>0$ (independent of $n$ and $t$) such that  \begin{equation}\label{mucken}
\vert \mathcal{H}_n(t)\vert\le\left\{\begin{array}{ll}
C(N^{\frac{1}{3}}+|N-t^2|)^{-\frac{1}{4}},&\text{ for }t^2\leq 2N;\\
C(e^{-\gamma t^2}),&\text{ for }t^2> 2N,
\end{array} \right.
\end{equation}
where  $N=2n+1$ for all  $n\geq 0$, see
\cite[Formula (2.4)]{Muckenhoupt}. Therefore, given $t\in \R$ and $n_0\in\N$ such that $t^2\leq 2(2n_0+1)$, there exists $C>0$ (independent of $n_0$ and $t$) such that  \begin{equation}\label{mucken2}|H_n(t)|\le C\biggl(\frac{e^{\frac{t^2}{2}}\sqrt{2^n n!}}{n^{\frac{1}{12}}}\biggr),\qquad n\geq n_0, \end{equation}
where we have applied the first inequality of Muckenhoupt estimates \eqref{mucken}.

Now we  consider briefly the following functions. The Dirichlet kernel $(d_t)_{t\in\R}$ is defined by $$d_t(s):=\frac{\sin(ts)}{\pi s}, \qquad s\in\R\setminus\{0\},$$ and  the Fej\'{e}r kernel $(f_t)_{t\in\R},$ by $$f_t(s):=\frac{1-\cos(ts)}{\pi s^2}, \qquad s\in\R\setminus\{0\}.$$  Note that $\lVert f_t\rVert_1 = \vert t\vert $ and $\displaystyle{f_t(s)=\int_0^t d_u(s)\,du},$ for $ t,\,s\in\R,$ see more details in \cite{Katnelson}. Observe that  Dirichlet and Fej\'{e}r kernels satisfy the conditions of Theorem \ref{ExpEsc} and we obtain their Hermite expansions.

\begin{proposition}\label{kernel} Fixed $s\in\R$, we have that:

 \begin{eqnarray*}
 d_t(s)&=&\sum_{n=1}^{\infty}\frac{(-1)^{n-1}s^{2n-2}e^{-\frac{s^2}{4}}}{2^{2n-1}(2n-1)!\pi}H_{2n-1}(t), \qquad t\in\R,\cr
 f_t(s)&=&\frac{1-e^{-\frac{s^2}{4}}}{\pi s^2}+\sum_{n=1}^{\infty}\frac{(-1)^{n-1}s^{2n-2}e^{-\frac{s^2}{4}}}{2^{2n}(2n)!\pi}H_{2n}(t), \qquad t\in\R.\cr
 \end{eqnarray*}
 \end{proposition}

 \begin{proof}As $d_t(s)=-d_{-t}(s)$ for $t,s\in \R$, we have that $c_{2n}(d_{(\cdot)}(s))=0$ for $n\ge 0$ and
 \begin{eqnarray*}c_{2n-1}(d_{(\cdot)}(s))&=&\frac{(-1)^{2n-1}}{2^{2n-1}(2n-1)!\sqrt{\pi}}\int_{-\infty}^{\infty}\frac{d^{2n-1}}{dt^{2n-1}}(e^{-t^2})\frac{\sin(ts)}{\pi s}\, dt
 \\
&=&\frac{(-1)^{n-1}}{2^{2n-1}(2n-1)!\sqrt{\pi}}\int_{-\infty}^{\infty}e^{-t^2}s^{2n-2}\frac{\cos(ts)}{\pi}\, dt=\frac{(-1)^{n-1}s^{2n-2}e^{-\frac{s^2}{4}}}{2^{2n-1}(2n-1)!\pi},
\end{eqnarray*} for $n\geq 1,$ and $s\in \R$. Using similar ideas, we show that $c_{2n-1}(f_{(\cdot)}(s))=0$ for $n\ge 1$,
\begin{eqnarray*}c_0(f_{(\cdot)}(s))&=&\frac{1-e^{-\frac{s^2}{4}}}{\pi s^2}, \qquad\cr
c_{2n}(f_{(\cdot)}(s))&=&\frac{(-1)^{n-1}s^{2n-2}e^{-\frac{s^2}{4}}}{2^{2n}(2n)!\pi},
\end{eqnarray*}for  $s\in \R$ and $n\geq 1$.
\end{proof}

\begin{definition}
For $n\in \NN\cup\{0\}$, we denote by $h_n$ the family of functions defined by
$$
h_n(t):=\frac{1}{2^n n!\sqrt{\pi}}e^{-t^2}H_n(t), \qquad t\in \R.
$$
\end{definition}
Note that $$h_n(t)=\frac{1}{\sqrt{2^n n!\pi^{\frac{1}{2}}}}e^{-\frac{t^2}{2}}\mathcal{H}_n(t)=\frac{(-1)^n}{2^n n!\sqrt{\pi}}\frac{d^n}{dt^n}(e^{-t^2})(t),\qquad t\in \R,$$ with $n=0,1,2,\ldots$. As the Hermite polynomials, these functions satisfy recurrence relations and differential equations. We collect some of them  in the following proposition which we will use later. We apply formulae \eqref{pro4} and \eqref{ecdif} to show  parts (i) and (ii). The complete proof is left to the reader.

\begin{proposition}\label{propHerm} The family of functions $\{ h_n \}_{n\ge 0}$  satisfies:
\begin{itemize}
\item[(i)] $2(n+1)h_{n+1}(t)=2th_n(t)-h_{n-1}(t),$ for $t\in \R$ and $n\ge 1$.
\item[(ii)] $h_n''(t)+2th_n'(t)+2(n+1)h_n(t)=0, $ for $t\in \R$ and $n\ge 0$.
\item[(iii)] $h_n^{(k)}=(-1)^k2^k(n+1)\ldots(n+k)h_{n+k}$ for $n, k\ge 0$.
\end{itemize}
\end{proposition}


\begin{theorem}\label{main2} Take $1\le p<\infty $.\begin{enumerate}

\item[(i)] The set of functions $ \{h_n\}_{n\ge 0} \subset L^p(\R)$ and
$$
\lVert h_n\rVert_{p}\leq  \frac{C_{p}}{\sqrt{2^n n!}} , \qquad n\ge 1.
$$



\item[(ii)] For $n\geq 1,$ $$\frac{n+1}{n} \displaystyle\max_{t\in\mathcal{Z}(H_n)}|h_{n+1}(t)|=\frac{1}{2n} \lVert h_{n-1} \rVert_{\infty}\le \lVert h_n \rVert_1$$ where $\mathcal{Z}(H_n)$ is the set of zeros of $H_n$.
\end{enumerate}
\end{theorem}
\begin{proof} (i)  Note that $|h_n|$ is an even function by \eqref{pro1}. Then we divide the integral in two parts to apply the Muckenhoupt estimates \eqref{mucken}. For $N=2n+1$, and $t^2\le 2N$, we have that $$\int_{0}^{\sqrt{2N}}|h_n(t)|^p\,dt\leq \frac{C_p}{(2^n n!\pi^{\frac{1}{2}})^{\frac{p}{2}}}\int_{0}^{\sqrt{2N}}e^{-\frac{pt^2}{2}}(N^{\frac{1}{3}}+|N-t^2|)^{-\frac{p}{4}}\,dt\leq \frac{C_p}{(2^n n!)^{\frac{p}{2}}n^{\frac{p}{12}}}.$$ On the other hand, we have that $$\int_{\sqrt{2N}}^{\infty}|h_n(t)|^p\,dt\leq \frac{C_{p,\gamma}}{(2^n n!\pi^{\frac{1}{2}})^{\frac{p}{2}}}\int_{\sqrt{2N}}^{\infty}e^{-\frac{pt^2}{2}}e^{-\gamma p t^2}\,dt\leq \frac{C_p}{(2^n n!)^{\frac{p}{2}}},$$
and we obtain the inequality for $1\le p<\infty$.





\item[(ii)]  For $n\ge 1$, note that $$|h_{n-1}(t)|\leq \int_t^{\infty}|h_{n-1}'(s)|\,ds\leq 2 n\lVert h_{n}\rVert_1,\quad t\in \R,$$
where we have applied Proposition \ref{propHerm} (iii) to get the second inequality. To show the first equality,  due to $\displaystyle\lim_{t\to \pm\infty}h_{n-1}(t)=0$, we obtain that
 \begin{eqnarray*} \lVert h_{n-1} \rVert_{\infty} &=&\displaystyle\max_{\{ t\in\R\ |\ (h_{n-1})'(t)=0 \}}\vert h_{n-1}(t)\vert
=\displaystyle\max_{\{ t\in\R^+\ |\ h_n(t)=0 \}}\vert h_{n-1}(t)\vert\cr &=&2(n+1)\displaystyle\max_{\{ t\in\R^+\ |\ h_n(t)=0 \}}\vert h_{n+1}(t)\vert =2(n+1) \displaystyle\max_{\{ t\in\R^+\ |\ H_{n}(t)=0 \}}\vert h_{n+1}(t)\vert,
\end{eqnarray*}
where we have used Proposition \ref{propHerm} (iii) and (i) to get the result.
\end{proof}

\begin{remark}\label{remark}{\rm For $p=1$, we apply the Cauchy-Schwartz inequality  and $\{\mathcal{H}_n\}_{n\geq 0}$ is an orthonormal basis on $L^2(\R)$ to get a direct proof:
 $$ \Vert h_n\Vert_1=\frac{1}{\sqrt{2^nn!\pi^{\frac{1}{2}}}}\int_{-\infty}^\infty e^{-\frac{t^2}{2}}|\mathcal{H}_n(t)|\,dt\leq
\frac{1}{\sqrt{2^n n!\pi^{\frac{1}{2}}}}\left(\int_{-\infty}^\infty e^{-t^2}\, dt\right)^{1\over 2}=\frac{1}{\sqrt{2^n n!}},$$
for $n\ge 0$.}
\end{remark}

Now we consider the expansions of certain functions                                                                                                                                                                                                                                                                                                                                                                                                                                                                                                                                                                                                                                                                                                                                                                                                                                                                                                                                                                                                                                                                                                                                                                                                                                                                                                                                                                                                                                                                                                                                                                                                                                                                                                                                                                                                                                                                                                                                                                                                                                                                                                                                                                                                                                                                                                                                                                                                                                                                                                                                                                                                                                                                                                                                                                                                                                                                                                                                                                                                                                                                                                                                                                                                                                                                                                                                                                                                                                                                                                                                                                                                                                                                                                                                                                                                                                                                                                                                                                                                                                                                                                                                                                                                                                                                                                                                                                                                                                                                                                                                                                                                                                                                                                                                                                                                                                                                                                                                                                                                                                                                                                                                                                                                                                                                                                                                                                                                                                                                                                                                                                                                                                                                                                                                                                                                                                                                                                                                                                                                                                                                                                                                                                                                                                                                                                                                                                                                                                                                                                                                                                                                                                                                                                                                                                                                                                                                                                                                                                                                                                                                                                                                                                                                                                                                                                                                                                                                                                                                                                                                                                                                                                                                                                                                                                                                                                                                                                                                                                                                                                                                                                                                                                                                                                                                                                                                                                                                                                                                                                             Hermite polynomials in spaces $L^p(\R)$. This result will be used to prove that the span$\{h_n \,\, \vert {n\ge 0}\}$ is dense for all $1\leq p<\infty$. Observe the set  $ \hbox{span}\{t^n e^{-t^2}\}_{n\ge 0}$ is dense in $L^1(\R)$ (\cite[Theorem 5.7.2]{Szego}) and in $L^2(\R)$ (\cite[Theorem 5.7.1]{Szego}).

\begin{theorem} \label{densi} Take $1\leq p < \infty.$
\begin{itemize}
\item[(i)] Let $\lambda\in\C$, and $\eta_{\lambda}(t):= e^{-t^2} e^{\lambda t}$ for $t\in\R$. Then we have that
$$\eta_{\lambda}=\sqrt{\pi}e^{\frac{\lambda^2}{4}}\displaystyle\sum_{n=0}^{\infty}\lambda^n h_n \hbox{ in $L^p(\R)$.}$$
\item[(ii)] The set span$\{h_n \,\, \vert {n\ge 0}\}$ is dense in $L^p(\R).$
\end{itemize}
\end{theorem}
\begin{proof} (i) By (\ref{expo}), we get that $e^{-t^2}e^{\lambda t}=\sqrt{\pi}e^{\frac{\lambda^2}{4}}\displaystyle\sum_{n=0}^{\infty}\lambda^n h_n(t)$ pointwise for  $\lambda \in\C$. This convergence is in $L^p(\R)$ since
$$\lVert \displaystyle\sum_{n=0}^{\infty}\lambda^n h_n \rVert_p\leq C_{p}\displaystyle\sum_{n=0}^{\infty}\frac{|\lambda|^n}{\sqrt{2^n n!}}\leq C_{p}\displaystyle\sum_{n=0}^{\infty}\frac{(|\lambda|\sqrt{e})^n}{\sqrt{2^n} n^{{n\over 2}+\frac{1}{4}}}< \infty, \qquad \lambda\in\C,$$
where we have applied Theorem \ref{main2} (i) and  Stirling's formula.

(ii) Using the part (i), it is enough to see that span$\{e^{-t^2} e^{\lambda i t} \,\, \vert {\lambda\in\R}\}$ is dense in $L^p(\R)$ to get the result. To do this, we apply the Hahn-Banach Theorem. Let $f\in L^q(\R)$ with $\frac{1}{p}+\frac{1}{q}=1$ such that  $$\int_{-\infty}^{\infty}f(t)e^{-t^2}e^{\lambda i t}\,dt=0, \qquad \lambda \in\R.$$  By  H\"{o}lder inequality,  $f \eta_0\in L^1(\R)$ and then $0=\mathcal{F}(f\eta_0)(-\frac{\lambda}{2\pi})$ for all $\lambda\in\R,$ where $$\mathcal{F}(g)(\xi):=\displaystyle\int_{-\infty}^{\infty}g(t)e^{-2\pi i \xi t}\,dt,\qquad \xi\in\R,\,g\in L^1(\R),$$ is the classical Fourier transform. Since the Fourier transform is injective in $L^1(\R)$, we conclude that $f= 0$.
\end{proof}

To show the main theorem of this section, Theorem \ref{expvect},  we adapt the  proof given in \cite[Theorem 2, p.71] {Lebedev} in the scalar case to the vector-valued setting. Note that this proof is based in a technical lemma which also has to be reproved to the vector-valued case \cite[Lemma, p.68] {Lebedev}. We give the statement of this new lemma  and we avoid its proof which runs parallel to the original one.
\begin{lemma} \label{lematec} Let $(X, \Vert \quad\Vert)$ be a Banach space and $\phi:\RR \to X$  a continuous function such that $ \displaystyle{\int_{-\infty}^{\infty}(1+t^2)e^{-t^2}\Vert \phi(t)\Vert^2 dt<\infty.}$ Then
$$
\lim_{n\to \infty}{n^{1\over 4}\over(2^n n!\sqrt{\pi})^{1\over 2}}\int_{-\infty}^{\infty}e^{-t^2}H_n(t)\phi(t)dt=0.
$$

\end{lemma}

We need to remind the differentiability of vector-valued functions to give a vector-valued version of classical Theorem \ref{ExpEsc}. Let  $f:\R\to X$ be a vector-valued function, we say that $f$ is differentiable at $t$ if exists $$\displaystyle\lim_{h\to 0}\frac{1}{h}(f(t+h)-f(t))$$ on $(X, \Vert \quad\Vert)$. In this case, $f$ is continuous at $t$, see more details, for example in \cite[Chapter 1]{ABHN}.

\begin{theorem}\label{expvect} Let $X$ be a Banach space and $f:\R\to X$ a differentiable function such that the integral $\displaystyle{\int_{-\infty}^{\infty}e^{-t^2}\lVert f(t)\rVert^2\,dt}$ is finite, then the series $\displaystyle\sum_{n=0}^{\infty}c_n(f) H_n(t),$ with $$c_n(f)=\int_{-\infty}^{\infty}h_n(t)f(t)\,dt,$$ converges pointwise to $f(t)$ for $t\in \RR.$
\end{theorem}

\begin{proof} By the Cauchy-Schwartz inequality, we get that  $c_n(f)\in X,$ $$\lVert c_n(f) \rVert\leq (\lVert \mathcal{H}_n \rVert_2)^{\frac{1}{2}}\left(\frac{1}{2^n n!\sqrt{\pi}}\displaystyle\int_{-\infty}^{\infty}e^{-t^2}\lVert f(t) \rVert^2dt\right)^{\frac{1}{2}}=\left(\frac{1}{2^n n!\sqrt{\pi}}\displaystyle\int_{-\infty}^{\infty}e^{-t^2}\lVert f(t) \rVert^2dt\right)^{\frac{1}{2}}<\infty, $$ where we have applied that $\{\mathcal{H}_n\}_{n\ge 0}$ is a orthonormal basis in $L^2(\R)$ and  $f$ satisfies the hypothesis.

Let $S_m(f)$ be the sum of the  first $m+1$ terms of the series, $$S_m(f)(t):=\sum_{n=0}^m c_n(f)H_n(t)=\displaystyle\int_{-\infty}^{\infty}e^{-y^2}K_{m}(t,y)f(y)\,dy,\qquad t\in \RR,$$ where $$K_m(t,y)=\displaystyle\sum_{n=0}^m\frac{1}{2^n n!\sqrt{\pi}}H_n(t)H_n(y), \qquad t,y\in \RR.$$
Note that $K_m(t,y)=K_m(y,t)$, $\displaystyle\int_{-\infty}^{\infty}e^{-y^2}K_m(t,y)\,dy=1$  and $$K_m(t,y)=\frac{H_{m+1}(t)H_m(y)-H_{m+1}(y)H_m(t)}{(t-y)2^{m+1}m!\sqrt{\pi}}, \qquad t\not=y \in \RR,$$ see these properties in \cite[p. 71-72]{Lebedev}.

Now, we write \begin{eqnarray*}
&\quad&S_m(f)(t)-f(t)=\int_{-\infty}^{\infty}e^{-y^2}K_m(t,y)(f(y)-f(t))\,dy \\
&\qquad&=\frac{1}{2^{m+1}m! \sqrt{\pi}}\left(H_{m}(t)\int_{-\infty}^{\infty}e^{-y^2}H_{m+1}(y)\phi(t,y)\,dy -H_{m+1}(t)\int_{-\infty}^{\infty}e^{-y^2}H_{m}(y)\phi(t,y)\,dy \right),
\end{eqnarray*} where $$\phi(t,y)=\frac{f(y)-f(t)}{y-t}, \qquad t\not=y \in \RR.$$ Observe that as function of $y,$  $\phi(t, \cdot)$ is continuous in $\R$ for any $t\in \RR$, $$\displaystyle\int_{-\infty}^{\infty}(1+y^2)e^{-y^2}\lVert\phi(t,y)\rVert^2\,dy<\infty, \qquad t\in\RR,$$  since $\phi(t,\cdot)$ is bounded in any neighborhood of $y=t,$ and for sufficiently large $b>t,$ \begin{eqnarray*}
\displaystyle\int_b^{\infty}(1+y^2)e^{-y^2}\lVert\phi(t,y)\rVert^2\,dy&=&O(1)\displaystyle\int_b^{\infty}e^{-y^2}(\lVert f(y) \rVert^2+\lVert f(t) \rVert^2)\,dy
\end{eqnarray*}
is finite by hypothesis. Similarly we estimate the integral  for the interval $(-\infty,b)$. Then we apply twice the Lemma \ref{lematec} to the function $\phi(t, \cdot)$ to get that
\begin{eqnarray*}
\displaystyle\lim_{m\to\infty} \frac{(m+1)^{\frac{1}{4}}}{(2^{m+1}(m+1)!\sqrt{\pi})^{\frac{1}{2}}}\int_{-\infty}^{\infty}e^{-y^2}H_{m+1}(y)\phi(t,y)\,dy &=&0,\cr
\displaystyle\lim_{m\to\infty}\frac{m^{\frac{1}{4}}}{(2^{m}m!\sqrt{\pi})^{\frac{1}{2}}}\int_{-\infty}^{\infty}e^{-y^2}H_{m}(y)\phi(t,y)\,dy&=&0.
\end{eqnarray*}

Finally, the following expressions remain bounded when $m\to\infty:$ $$
\frac{(2^{m+1}(m+1)!\sqrt{\pi})^{\frac{1}{2}}}{(m+1)^{\frac{1}{4}}}\frac{H_m(t)}{2^{m+1}m!\sqrt{\pi}},\ \ \frac{(2^{m}m!\sqrt{\pi})^{\frac{1}{2}}}{m^{\frac{1}{4}}}\frac{H_{m+1}(t)}{2^{m+1}m!\sqrt{\pi}},
$$
using \cite[(4.14.9), p.67]{Lebedev} and Stirling's formula. We conclude that $\displaystyle\lim_{m\to\infty}\lVert S_m(f)(t)-f(t)\rVert=0,$ for $t\in \RR$.
\end{proof}

\begin{remark} {\rm A straighforward application of Theorem \ref{ExpEsc} allows to obtain the weak convergence of the partial serie $S_m(f)(t)$ to the function $f(t)$ for $t\in \RR$.}

\end{remark}

\section{Hermite expansions for $C_0$-groups}

\setcounter{theorem}{0}
\setcounter{equation}{0}

In this section, we approximate $C_0$-groups by their Hermite expansions. A different approach, using stable rational approximations, is posed in \cite{BT}. The rate of convergence depends on the smoothness of initial data in both cases, compare  \cite[Theorem 3]{BT} and  Theorem \ref{rate1} (ii). We also give  a new serie representation for holomorphic semigroups (Theorem \ref{loc} (iii)) and the Hermite expansion  for F\'{e}jer operators (Corollary \ref{corola}),  both families of operators are subordinated to the initial $C_0$-group.

First of all, we give some basic results from $C_0$-group theory. Given $(A, D(A))$ a closed operator  on a Banach space $X,$ the resolvent operator $\lambda\to (\lambda-A)^{-1}$ is analytic in the resolvent set, and $$\frac{d^n}{d\lambda^n}(\lambda-A)^{-1}x=(-1)^n n! (\lambda-A)^{-n-1}x\ \text{ for all }n\in\N,\quad x\in X,$$ see \cite[p.240]{Nagel}.
  It is known that $A$ generates a $C_0$-group, $(T(t))_{t\in\R}$, if and only if $\pm A$ generates a $C_0$-semigroup. Every $C_0$-group is exponentially bound, i.e., there exist $M, \omega\ge 0$ such that $\lVert T(t)\rVert\leq M e^{w|t|}$ for all $t\in\R.$ For $\alpha>0$ and $\Re(\lambda)>w,$  the fractional powers of the resolvent operator is defined by
\begin{equation*}\label{resolvent}
(\lambda\mp A)^{-\alpha}x:=\frac{1}{\Gamma(\alpha)}\int_0^{\infty}t^{\alpha-1}e^{-\lambda t}T(\pm t)x\,dt, \qquad x\in X, \end{equation*}
see \cite[Proposition 11.1]{Komatsu}. Note that the operator $(A^2, D(A^2))$ generates a holomorphic $C_0$-semigroup of angle $\frac{\pi}{2}$, $$T^{(g)}(z)x:=\frac{1}{\sqrt{4\pi z}}\int_{-\infty}^{\infty}e^{\frac{-t^2}{4z}}T(t)x\,dt,\qquad \Re(z)>0,   $$ for  $x\in X,$ see \cite[Example 3.14.15 and Theorem 3.14.17]{ABHN}. We say that a $C_0$-group is uniformly bounded if $\lVert T(t)\rVert\leq M$ for all $t\in \R.$  For further details see, for example, monographies \cite{ABHN, Nagel}.

\begin{theorem}\label{loc} Let $(T(t))_{t\in\R}$ be a $C_0$-group on a Banach space  $X$ with infinitesimal generator $(A, D(A)).$
 \begin{itemize}

 \item[(i)]For $n\in \NN\cup\{0\}$, we get that
 $$
  \int_{-\infty}^{\infty}h_n(t)T(t)x\,dt=\frac{1}{2^n n!}A^n T^{(g)}\biggl(\frac{1}{4}\biggr)x, \qquad x\in X.
  $$ In the case that $\sup_{t\in \R}\Vert T(t)\Vert<\infty$, we have that
  $$
  \Vert A^n T^{(g)}\biggl(\frac{1}{4}\biggr)x\Vert \le \left(\sup_{t\in \R}\Vert T(t)\Vert\right)\sqrt{2^n n!}\,\Vert x\Vert, \qquad n\in \NN,\quad x\in X.
  $$
   .

 \item[(ii)] For $x\in D(A),$ the following equality holds: $$T(t)x=\displaystyle\sum_{n=0}^{\infty}\frac{1}{2^{n} n!}A^nT^{(g)}\biggl(\frac{1}{4}\biggr)x\,H_{n}(t), \qquad t\in\R.$$

\item[(iii)] In the case that $\sup_{t\in \R}\Vert T(t)\Vert <\infty$, we have that
$$T^{(g)}(z)x=\displaystyle\sum_{n=0}^{\infty}\frac{1}{2^{2n} n!}A^{2n}T^{(g)}\biggl(\frac{1}{4}\biggr)x(4z-1)^n, \qquad \vert z-{1\over 4}\vert<{1\over 4},$$
for $x\in X$.
 \end{itemize}

\end{theorem}
\begin{proof} (i)  For $n\in \NN\cup\{0\}$, we obtain

$$\int_{-\infty}^{\infty}h_n(t)T(t)x\,dt=\frac{(-1)^n}{2^nn!\sqrt{\pi}}\int_{-\infty}^{\infty}\frac{d^n}{dt^n}(e^{-t^2})T(t)x\,dt=\frac{A^n}{2^nn!\sqrt{\pi}}\int_{-\infty}^{\infty}e^{-t^2}T(t)x\,dt, $$ for $x\in X$, where we have  integrated by parts. Then the result is obtained. In the case that $\sup_{t\in \R}\Vert T(t)\Vert <\infty$, we apply  Remark \ref{remark} to get the bound.


(ii) Note that  $\int_{-\infty}^{\infty}e^{-t^2}\lVert T(t)x\rVert^2\,dt<\infty,$ for $x\in X$, the function $T(\cdot)x:\R\to X$ is differentiable at every point and ${d\over dt}T(t)x=T(t)Ax$, for $t\in \R$ and $x\in D(A)$. Then, we apply Theorem \ref{expvect} to get $$\lVert T(t)x-\displaystyle\sum_{n=0}^{m}\frac{1}{2^n n!}A^n T^{(g)}\biggl(\frac{1}{4}\biggr)x\,H_n(t) \rVert\to 0, \qquad x\in D(A),$$ as $m\to\infty,$ for all $t\in\R$.

(iii) By the part (i) and Stirling's formula, we have that
$$
\displaystyle\sum_{n=0}^{\infty} \frac{\Vert A^{2n}T^{(g)}(\frac{1}{4})x\Vert}{2^{2n} n!}\vert 4z-1\vert^n\le \left(\sup_{t\in \R}\Vert T(t)\Vert\right)\Vert x\Vert\displaystyle\sum_{n=0}^{\infty} \frac{\sqrt{(2n)!}}{2^{n} n!}{\vert 4z-1\vert^n}\le M'\sum_{n=0}^{\infty} {\vert 4z-1\vert^n\over n^{1\over 4}}<\infty,
$$
for $x\in X$ and $  \vert z-{1\over 4}\vert<{1\over 4}$. Then we define the analytic family of operators $$F(z)x:=\displaystyle\sum_{n=0}^{\infty}\frac{A^{2n}T^{(g)}(\frac{1}{4})x}{2^{2n} n!}(4z-1)^n, \qquad \vert z-{1\over 4}\vert<{1\over 4}, \quad x\in X.
$$
Note that for $x\in D(A^2)$, we have that $$F'(z)x= \displaystyle\sum_{n=1}^{\infty}\frac{A^{2n}T^{(g)}(\frac{1}{4})x}{2^{2n} n!}4n(4z-1)^{n-1}=\sum_{n=0}^{\infty}\frac{A^{2n}T^{(g)}(\frac{1}{4})(A^2(x))}{2^{2n} n!}(4z-1)^{n-1}= F(z)A^2(x).
$$
Since $A^2$ generates a holomorphic $C_0$-semigroup of angle $\frac{\pi}{2}$ and $F(\frac{1}{4})=T^{(g)}(\frac{1}{4})$ then we have that $F(t)=T^{(g)}(t)$ for $t\in (0,\frac{1}{2})$ by the uniqueness of the solution of the first order abstract Cauchy problem, see for example \cite[Corollary 3.7.21]{ABHN}. Applying the principle of analytic continuation we conclude that $F(z)x= T^{(g)}(z)x$ for $x\in X$ and $\vert z-{1\over 4}\vert<{1\over 4}$.
\end{proof}


\begin{remark}\label{cpm}{\rm It might be of interest to investigate for which $0<\alpha
<1$ one has norm convergence for $x\in D((-A)^\alpha)$, see Theorem \ref{loc} (ii). Different ways might be followed: more general versions of Theorem \ref{ExpEsc}; inequalities obtained from Bernstein's Lemma for Fourier multipliers (\cite[Lemma 8.2.1 and Proposition 8.2.3]{ABHN}) or functional calculus defined via transference (\cite{HR}). However this is not the main aim of this paper, and we leave this question in this point.

Now we  compare Hermite expansions of $C_0$-groups and Laguerre expansions of $C_0$-semigroups studied in \cite{Abadias} . We denote by $L_n^{(\alpha)}$ the usual Laguerre polynomial of degree $n\in \NN$ and parameter $\alpha \in \R$ and
 $$H_{2n}(t)=(-1)^n2^{2n}n! L_n^{(-\frac{1}{2})}(t^2)\text{ and }H_{2n+1}(t)=(-1)^n2^{2n+1}n!t L_n^{({1\over 2})}(t^2), \qquad t\in \R,$$
 see for example \cite[Formula 4.19.5]{Lebedev}. Let  $(T(t))_{t\in\R}$ be a uniformly bounded $C_0$-group with generator $(A, D(A)).$ By Theorem \ref{loc} (ii), we  may express the $C_0$-group $(T(t))_{t\in \R}$ in terms of Hermite polynomials and also in terms of Laguerre polynomials:
\begin{eqnarray*}
T(t)x&=&\lim_{N\to \infty}\left(\sum_{n=0}^{2N+1}\frac{1}{2^{n} n!}A^nT^{(g)}\biggl(\frac{1}{4}\biggr)x\,H_{n}(t)\right) \cr
&=&\lim_{N\to \infty}\left(\sum_{j=0}^{N}\frac{(-1)^jj!}{(2j)!}A^{2j}T^{(g)}\biggl(\frac{1}{4}\biggr)\left(x L_j^{(-\frac{1}{2})}(t^2)+{Ax\over 2j+1}t L_n^{({1\over 2})}(t^2)\right)\right),
\end{eqnarray*}
for $t\in \RR$ and $x\in D(A)$.
 Although other Laguerre expansions are obtained in \cite[Theorem 4.1.(iii)]{Abadias}, it seems more natural to express a $C_0$-group in terms of its Hermite expansion due to  the order of convergence is sharper, see Remark \ref{compa}.
}
\end{remark}

As we have commented, it is important to know the rate of approximation of truncated Hermite expansion to the $C_0$-group. This theory has a key role in many areas of maths, as in PDE's, harmonic analysis or numerical analysis. Before to calculate the order of convergence in  Theorem \ref{rate1}, we prove the following technical lemma.


\begin{lemma}\label{lema}Let $(T(t))_{t\in\R}$ be a uniformly bounded $C_0$-group in a Banach space $X$  with infinitesimal generator $(A,D(A))$ and $p$ a positive integer number. Then \begin{equation}\label{lemaeq}\lVert\int_{-\infty}^{\infty}h_n(t)T(t)x\,dt\rVert \leq\frac{C\sqrt{(n-p)!}}{2^{\frac{n+p}{2}}n!}\lVert A^p x \rVert,\end{equation} for $x\in D(A^p)$, $n\geq p$  and $C$ a positive constant.
\end{lemma}
\begin{proof}
We define by $B(x):=\frac{1}{2^n n!}A^nT^{(g)}(\frac{1}{4})x=\int_{-\infty}^{\infty}h_n(t)T(t)x\,dt,$ for $x\in X,$ see Theorem \ref{loc} (i). For $x\in D(A^p)$, we apply Proposition \ref{propHerm} (iii) and integrate by parts to get that  \begin{eqnarray*}
B(x)&=&\frac{(-1)^p}{2^pn\ldots(n-p+1)}\int_{-\infty}^{\infty} \frac{d^p}{dt^p}(h_{n-p})(t)T(t)x\,dt\cr
&=&\frac{1}{2^pn\ldots(n-p+1)}\int_{-\infty}^{\infty}h_{n-p}(t)A^pT(t)x\,dt,\end{eqnarray*}
where we have used that $\lim_{t\to\pm\infty}\frac{d^{p-j}}{dt^{p-j}}(h_{n-p})(t)=0,$ for $j=1,2,\ldots,p.$ By Theorem \ref{main2} (i) for $p=1$,  we conclude that $$\lVert B(x) \rVert\leq \frac{1}{2^pn\ldots(n-p+1)}\lVert h_{n-p}\rVert_1\,\displaystyle\sup_{t\in\R}\lVert A^p T(t)x \rVert\leq\frac{C\sqrt{(n-p)!}}{2^{\frac{n+p}{2}}n!}\lVert A^p x \rVert$$ for $x\in D(A^p).$
\end{proof}

We define the $m$-th partial sum of the Hermite expansion, $$T_{m}(t)x:=\displaystyle\sum_{n=0}^{m}\frac{1}{2^n n!}A^n T^{(g)}(\frac{1}{4})x\,H_n(t), \qquad x\in X,\quad t\in \R,$$
for $m\ge 0$, and  we are interested to estimate the  rate of convergence of the $m$-th partial sum $(T_m(t))_{ t\in \R}$ to the  $C_0$-group $(T(t))_{t\in \R}$.

\begin{theorem}\label{rate1} Let $(T(t))_{t\in\R}$ be a uniformly bounded $C_0$-group on a Banach space $X$ with infinitesimal generator $(A,D(A)).$ Then for each $t\in\R$ there is a $m_0\in\N$ such that for $m\geq m_0$ and   $2\leq p\leq m+1$,   $$\lVert T(t)x-T_{m}(t)x \rVert\leq \frac{C_{t,p}}{m^{\frac{p}{2}-\frac{11}{12}}}\lVert A^p x\rVert, \qquad  x\in D(A^p).$$ Moreover, the convergence is  locally uniformly in $t.$
\end{theorem}
\begin{proof} Fix  $t\in\R$. By  inequalities (\ref{mucken2}) and \eqref{lemaeq}, there is a $m_0\in\N$ such that for $m\geq m_0,$ \begin{eqnarray*}
\lVert T(t)x-T_{m}(t)x\rVert&\leq&\displaystyle\sum_{n=m+1}^{\infty}\lVert \int_{-\infty}^{\infty}h_n(s)T(s)x\,ds \rVert |H_n(t)|  \leq\displaystyle\sum_{n=m+1}^{\infty}\frac{C_t}{2^{\frac{p}{2}}n^{\frac{1}{12}}}\sqrt{\frac{(n-p)!}{n!}}\lVert A^p x\rVert \cr
&\leq& \displaystyle\sum_{n=m+1}^{\infty}\frac{C_t}{2^{\frac{p}{2}}n^{\frac{p}{2}+\frac{1}{12}}}\lVert A^p x\rVert
\leq\frac{C_{t,p}}{m^{\frac{p}{2}-\frac{11}{12}}}\lVert A^p x\rVert,
\end{eqnarray*}  for $x\in D(A^p)$. The locally uniformly convergence is clear, since if $K\subset \R$ is compact it is sufficient to take the maximum of all $m_0$ ( a finite number of $m_0$ due to the compactness of $K$) for which the pointwise convergence is satisfied.
\end{proof}

\begin{remark}\label{compa}{\rm
Note that for $t>0$, we consider the $m$-th partial sum of Laguerre expansion $S_{m,\alpha}(t)x=\displaystyle\sum_{n=0}^{m}(-A)^n(1-A)^{-n-\alpha-1}x L_n^{(\alpha)}(t),$ ($x\in X$) introduced in \cite{Abadias}, compare with Remark \ref{cpm}. By  \cite[Theorem 5.2]{Abadias}, we have that
$$\lVert T(t)x-S_{m,\alpha}(t)x \rVert\leq \frac{C_{t,p}}{m^{\frac{p}{2}-1}}\lVert A^p x\rVert, \qquad x\in D(A^p),$$
 for $t>0$, $2<p\le m+1$ and therefore Hermite approximation is sharper than Laguerre approximation, compare with  Theorem \ref{rate1}}.
\end{remark}

Finally, let  $T\equiv(T(t))_{t\in\R}$ be a uniformly bounded $C_0$-group   on a Banach space $X$ and we consider the family of operators
\begin{eqnarray*}\mathcal{D}^{T}_N(t)x:&=&\displaystyle\int_{-N}^{N}d_t(s)T(s)x\,ds,\cr
\mathcal{F}^{T}(t)x:&=&\int_{-\infty}^{\infty}f_t(s)T(s)x\,ds,
\end{eqnarray*}
 for $x\in X$ and $N\ge 1$ (functions $d_t$ and $f_t$ are considered in Section 1). The uni-parametric family  $(\mathcal{F}^{T}(t))_{t\in \R}$  is  a Fej\'{e}r family and is studied in detail in \cite[Section 5]{Fasangova}. In the case that $X$ is a UMD Banach space then there exists $\lim_{N\to +\infty}\mathcal{D}^{T}_N(t)x$ for $x\in X$ and defined  bounded operators, called $(\mathcal{D}^{T}(t))_{t\in \R}$ a Dirichlet family, see \cite[Section 4]{Fasangova}.

\begin{corollary}\label{corola} Let $X$ be a  Banach space and $T\equiv(T(t))_{t\in\R}$ a uniformly bounded $C_0$-group on $X$. For $x\in X$ and $t\in \R$, we have that
\begin{equation}\label{fejer2}\mathcal{F}^{T}(t)x=\int_{-\infty}^{\infty}\frac{1-e^{-\frac{s^2}{4}}}{\pi s^2}T(s)x\,ds +
\displaystyle\sum_{n=1}^{\infty}\biggl(\int_{-\infty}^{\infty}\frac{(-1)^{n-1}s^{2n-2}e^{-\frac{s^2}{4}}}{2^{2n}(2n)!\pi}T(s)x\,ds\biggr)H_{2n}(t).\end{equation}
\end{corollary}

\begin{proof} We use Proposition \ref{kernel} to express the F\'{e}jer kernel in its Hermite expansion. We need to commute the integral and the serie.
Note that the vector-valued coefficients of the expansion (\ref{fejer2}) are bounded by $\lVert c_{2n}(f_{(\cdot)})\rVert_1,$ where $c_{2n}(f_{(\cdot)})$ is the coefficient  of Hermite expansion  of F\'{e}jer kernel given in Proposition \ref{kernel}. Then $$\lVert c_{2n}(f_{(\cdot)})\rVert_1=\frac{1}{2^{2n-1}(2n)!\pi}\int_0^{\infty}s^{2(n-1)}e^{-\frac{s^2}{4}}\,ds=\frac{(2n-3)!!}{2^{n-1}(2n)!\sqrt{\pi}}=\frac{1}{(2n-1)2^{2n-1}n!\sqrt{\pi}},$$ where we have applied that $$\displaystyle\int_0^{\infty}x^{2n}e^{-px^2}\,dx=\frac{(2n-1)!!}{2(2p)^n}\sqrt{\frac{\pi}{p}},\qquad p>0, n=0,1,\ldots,$$ see \cite[Section 3.461, formula 2, p. 360]{Gradshteyn}. Then we have that $$\lVert c_{2n}(f_{(\cdot)})\rVert_1\,|H_{2n}(t)|\leq \frac{C_t}{n^{\frac{4}{3}}},$$ where we have used Muckenhoupt estimates \eqref{mucken2} and Stirling's formula. Therefore  the serie is absolutely convergent and we obtain the result.
\end{proof}

\begin{remark} {\rm In the case that  $X$ is a UMD Banach space, $\mathcal{D}^{T}(t)x={d\over dt }\mathcal{F}^{T}(t)x$ holds for $x\in X.$  However we do not conclude that
\begin{equation}\label{dirich}
\mathcal{D}^{T}(t)x= \sum_{n=1}^{\infty}\biggl(\int_{-\infty}^{\infty}
\frac{(-1)^{n-1}s^{2n-2}e^{-\frac{s^2}{4}}}{2^{2n-1}(2n-1)!\pi}T(s)x\,ds\biggr)H_{2n-1}(t),\quad x\in X,
\end{equation}
for $t\in \R$ due to
$$\lVert c_{2n-1}(d_{(\cdot)})\rVert_1\,|H_{2n-1}(t)|\leq \frac{C_t}{n^{\frac{5}{6}}}.$$
Additional ideas are needed to assure the convergence of this conjectured Hermite serie. It is known that  convergence rates for approximations may be improved in UMD and Hilbert spaces. Some nice inequalities, proved recently in \cite{HR}, might be applied in these families of spaces to conclude the convergence of the right part of (\ref{dirich}).  }
\end{remark}

\section{ Hermite expansions for cosine functions}

\setcounter{theorem}{0}
\setcounter{equation}{0}

In this section we consider operator families related to the second order Cauchy problem, the cosine and sine functions. We obtain their Hermite expansions (Theorem \ref{loc2} (ii)) and  a new  representation serie for subordinated holomorphic semigroup (Theorem \ref{loc2} (iii)).  We present that Hermite expansions for $C_0$-group and Hermite expansions for cosine functions are consistent, see Remark \ref{conn}. We also give the rate of convergence of the truncated Hermite expansion to the cosine function, Theorem \ref{rate2} (iii).

Let $X$ be a Banach space and $C:\R\to \mathcal{B}(X)$ a cosine function, see definition in the Introduction or \cite[Section 3.14]{ABHN}. All cosine functions are even, exponentially bounded, $\lVert C(t)\rVert\leq Me^{wt}$ with $M\geq 1$ and $w\geq 0,$ and $$\lambda(\lambda^2-A)^{-1}x=\int_0^{\infty}e^{-\lambda t}C(t)x\,dt,\qquad \Re(\lambda)>w,x\in X,$$ where $A$ is the generator of $(C(t))_{t\in\R}.$ The Euler approximation for cosine functions (for $C_0$-semigroups, see for example \cite[Corollary 3.3.6]{ABHN} and \cite{Gomilko}),
$$
C(t)x=\lim_{n\to \infty}{(-1)^n}{1\over n!}\left({n\over t}\right)^{n+1}\left(\lambda(\lambda^2-A)^{-1}\right)^{(n)}\left({n\over t}\right)x, \qquad x\in X,\quad t\in \R,$$
is a consequence of the Post-Widder inversion formula for the Laplace transform (see for example \cite[Theorem 1.7.7]{ABHN}). Other different approach is followed in \cite[Section 4]{jara} where stable rational schemes are used to approximate cosine functions.

 Observe that $A$ generates a holomorphic $C_0$-semigroup of angle $\frac{\pi}{2}$, $$T^{(c)}(z)x=\frac{1}{\sqrt{\pi z}}\int_{0}^{\infty}e^{\frac{-t^2}{4z}}C(t)x\,dt,\qquad \Re(z)>0,  $$ for  $x\in X,$ see for example \cite[Corolary 3.14.17]{ABHN}. There are other operators associated to the second order Cauchy problem. The sine function $S:\R\to\mathcal{B}(X)$ associated with $C$ is defined by $$S(t)x:=\int_0^{t}C(s)x\,ds \qquad t\in\R,x\in X.$$ Observe that $S$ is an odd function; see more details in \cite[Section 3.14]{ABHN}.




\begin{theorem}\label{loc2} Let $(C(t))_{t\in\R}$ be a cosine function on a Banach space $X$ with infinitesimal generator $(A, D(A)).$
 \begin{itemize}

 \item[(i)]For $n\in \NN\cup\{0\}$, we get that
 $$
  \int_{-\infty}^{\infty}h_{2n+1}(t)C(t)x\,dt=\int_{-\infty}^{\infty}h_{2n}(t)S(t)x\,dt=0, \qquad x\in X,
 $$
 $$
  \int_{-\infty}^{\infty}h_{2n}(t)C(t)x\,dt=\frac{1}{2^{2n} (2n)!}A^n  T^{(c)}\biggl(\frac{1}{4}\biggr)x, \qquad x\in X,
 $$ and
 $$
  \int_{-\infty}^{\infty}h_{2n+1}(t)S(t)x\,dt=\frac{1}{2^{2n+1} (2n+1)!}A^n  T^{(c)}\biggl(\frac{1}{4}\biggr)x, \qquad x\in X.
 $$ 

 In the case that $\sup_{t\in \R}\Vert C(t)\Vert<\infty$, we have that
  $$
  \Vert A^n T^{(c)}\biggl(\frac{1}{4}\biggr)x\Vert \le \sup_{t\in \R}\Vert C(t)\Vert\, 2^{n}\sqrt{ (2n)!}\,\Vert x\Vert, \qquad n\in \NN,\quad x\in X.
  $$
 \item[(ii)] For $x\in D(A),$  the following Hermite expansions hold: \begin{equation*}\label{Cos}C(t)x=\displaystyle\sum_{n=0}^{\infty}\frac{1}{2^{2n}(2n)!}A^n T^{(c)}\biggl(\frac{1}{4}\biggr)x\,H_{2n}(t), \qquad t\in\R,\end{equation*} and \begin{equation*}\label{Sen}S(t)x=\displaystyle\sum_{n=0}^{\infty}\frac{1}{2^{2n+1}(2n+1)!}A^n T^{(c)}\biggl(\frac{1}{4}\biggr)x\,H_{2n+1}(t), \qquad t\in\R.\end{equation*}

\item[(iii)] In the case that $\sup_{t\in \R}\Vert C(t)\Vert <\infty$, the following series representation holds:
 $$
 T^{(c)}(z)x=\sum_{n=0}^{\infty} \frac{1}{2^{2n}n!}A^n T^{(c)}\biggl(\frac{1}{4}\biggr)x (4z-1)^n, \qquad \vert z-{1\over 4}\vert<{1\over 4},$$
for $x\in X$.
 \end{itemize}

\end{theorem}
\begin{proof} (i) Since  $h_{2n+1}$ is an odd function and $h_{2n}$ is even, first identities are shown  for $n\in \N\cup \{0\}$. On the other hand,
$$\int_{-\infty}^{\infty}h_{2n}(t)C(t)x\,dt=\frac{1}{2^{2n}(2n)!\sqrt{\pi}}\int_{-\infty}^{\infty}\frac{d^{2n}}{dt^{2n}}(e^{-t^2})C(t)x\,dt=2\frac{A^n}{2^{2n}(2n)!\sqrt{\pi}}\int_{0}^{\infty}e^{-t^2}C(t)x\,dt,$$ for $x\in X,$ and \begin{eqnarray*}\int_{-\infty}^{\infty}h_{2n+1}(t)S(t)x\,dt
&=&\frac{1}{2^{2n+1}(2n+1)!\sqrt{\pi}}\int_{-\infty}^{\infty}\frac{d^{2n}}{dt^{2n}}(e^{-t^2})C(t)x\,dt \\
&=&2\frac{A^n}{2^{2n+1}(2n+1)!\sqrt{\pi}}\int_{0}^{\infty}e^{-t^2}C(t)x\,dt.\end{eqnarray*}



(ii) Note that  $\int_{-\infty}^{\infty}e^{-t^2}\lVert C(t)x\rVert^2\,dt<\infty $ for $x\in X$;  the function $C(\cdot)x:\R\to X$ is twice differentiable at every point and ${d^2\over dt^2}C(t)x=C(t)Ax$ for  $x\in D(A),$ see \cite[Proposition 3.14.5 a)]{ABHN}.  Then we apply Theorem \ref{expvect} and we have that $$\lVert C(t)x-\displaystyle\sum_{n=0}^{m}\frac{1}{2^{2n}(2n)!}A^n T^{(c)}\biggl(\frac{1}{4}\biggr)x\,H_{2n}(t) \rVert\to 0$$ as $m\to\infty,$ for all $t\in\R .$ . By the fundamental theorem of calculus, $S(\cdot)x:\R\to X$ is also differentiable at every point, and the rest is a simple check.

(iii) The proof runs parallel to the proof of Theorem \ref{loc}(iii).
\end{proof}

\begin{remark}\label{conn}{\rm As Example 5.3 shows, the bound given in Theorem \ref{loc2}(i) is not optimal in some particular cases.

Now we show that Hermite expansions of cosine functions is related with Hermite expansions of $C_0$-groups. It is well-know that if $A$ is the generator of a $C_0$-group $(T(t))_{t\in \RR}$ on a Banach space $X$, then $A^2$ generates a cosine function in $X$ given by
$$
C(t)={T(t)+T(-t)\over 2}, \qquad t\in \R,
$$
see for example \cite[Example 3.14.15]{ABHN}. We apply Theorem \ref{loc} (ii)  to get that $$
C(t)x=\sum_{n=0}^{\infty}\frac{1}{2^{2n}(2n)!}A^{2n} T^{(c)}\biggl(\frac{1}{4}\biggr)x\,H_{2n}(t), \qquad t\in \R,
$$
for $x\in D(A)$ (note that $T^{(c)}({1\over 4})= T^{(g)}({1\over 4})$).  In this particular case, we improve  Theorem \ref{loc2} (ii) due to the equality holds for $x\in D(A)$ (larger than the set $D(A^2)$). In addtion, observe that Theorem \ref{loc} (iii) is a direct consequence of Theorem \ref{loc2} (iii).

Conversely if $A$ generates a uniformly bounded cosine function $(C(t))_{t\in \R}$ in a UMD Banach space $X$, then $i(-A)^{1\over 2}$ generates a $C_0$-group, $(\widetilde{T}(t))_{t\in \R}$, where
 $$
 \widetilde{T}(t):= C(t)+i(-A)^{1\over 2}S(t), \qquad t\in \RR,$$
 see for example \cite[Theorem 3.16.7]{ABHN} and \cite[Corollary 2.6]{CK}. In fact the $C_0$-group, $(\widetilde{T}(t))_{t\in \R}$ is uniformly bounded as it is proved in \cite[Theorem 1.1]{H}. By Theorem \ref{loc2} (ii), we have that
 \begin{eqnarray*}
 \widetilde{T}(t)x&=&  \displaystyle\sum_{n=0}^{\infty}\frac{1}{2^{2n}(2n)!}A^n T^{(c)}({1\over 4})x\,H_{2n}(t)+ i(-A)^{1\over 2}\displaystyle\sum_{n=0}^{\infty}\frac{1}{2^{2n+1}(2n+1)!}A^n T^{(c)}({1\over 4})x\,H_{2n+1}(t)\cr
 &=&\displaystyle\sum_{n=0}^{\infty}\frac{1}{2^{n}n!}(i(-A)^{1\over 2})^n \tilde{T}^{(g)}\biggl(\frac{1}{4}\biggr)x\,H_{n}(t)
 \end{eqnarray*}
for $x\in D(A)$, $t\in \RR$ (again  $T^{(c)}({1\over 4})= \tilde{T}^{(g)}({1\over 4})$). Note that in this case Theorem \ref{loc} (ii) is a extension of Theorem \ref{loc2} (ii) due to $D(A)\subset D((-A)^{1\over 2})$.}

\end{remark}

Following the same type of arguments as in the previous section, we give the order of convergence of truncated Hermite expansions $(C_{m}(t))_{t\in \R},(S_{m}(t))_{t\in \R}$  where $$C_{m}(t)x:=\displaystyle\sum_{n=0}^{m}\frac{1}{2^{2n} (2n)!}A^n T^{(c)}\biggl(\frac{1}{4}\biggr)x\,H_{2n}(t)$$ and $$S_{m}(t)x:=\displaystyle\sum_{n=0}^{m}\frac{1}{2^{2n+1}(2n+1)!}A^n T^{(c)}\biggl(\frac{1}{4}\biggr)x\,H_{2n+1}(t),$$ for $t \in \R,$ $x\in X$ and $m\ge 0,$ to the cosine and sine functions, $(C(t))_{t\in \R}$ and $(S(t))_{t\in \R}$ respectively. We also give the Hermite expansion of subordinated F\'{e}jer families, $(\mathcal{F}^{C}(t))_{t\in \R}$, where
$$\mathcal{F}^{C}(t)x:=2\int_{0}^{\infty}f_t(s)C(s)x\,ds, \qquad x\in X,$$  see \cite[Corollary 5.5]{Fasangova}.

\begin{theorem}\label{rate2}
Let $(C(t))_{t\in\R}$ be a uniformly bounded cosine funtion on a Banach space $X$ with infinitesimal generator $(A,D(A))$.

 \begin{itemize}
 \item[(i)] Let $p$ be a positive integer. Then for $x\in D(A^p)$ and $n\geq p,$ we get $$\label{lema2}\lVert\int_{-\infty}^{\infty}h_{2n}(t)C(t)x\,dt\rVert \leq\frac{C\sqrt{(2n-2p)!}}{2^{n+p}(2n)!}\lVert A^p x \rVert,$$ with $C$ a positive constant.

 \item[(ii)] Let $p$ be a integer greater or equal than zero. Then for $x\in D(A^{p+1})$ and $n\geq p,$ we get $$\label{lema2}\lVert\int_{-\infty}^{\infty}h_{2n+1}(t)S(t)x\,dt\rVert \leq\frac{C\sqrt{(2n-2p)!}}{2^{n+p+1}(2n+1)!}\lVert A^{p+1} x \rVert,$$ with $C$ a positive constant.

 \item[(iii)] Then for each $t\in\R$ there is a $m_0\in\N$ such that for $m\geq m_0,$ and $1\leq p\leq m+1$, we have that $$\lVert C(t)x-C_{m}(t)x \rVert\leq \frac{C_{t,p}}{m^{p-\frac{11}{12}}}\lVert A^p x\rVert, \qquad x\in D(A^p),$$ and $$\lVert S(t)x-S_{m}(t)x \rVert\leq \frac{C_{t,p}}{m^{p-\frac{5}{12}}}\lVert A^{p+1} x\rVert, \qquad x\in D(A^{p+1}).$$ Moreover, the convergence is locally uniformly in t.

\item[(iv)] For $x\in X$ and $t\in \R$, the following equality holds: $$\mathcal{F}^{C}(t)x=2
\int_{0}^{\infty}\frac{1-e^{-\frac{s^2}{4}}}{\pi s^2}C(s)x\,ds +
2\displaystyle\sum_{n=1}^{\infty}\biggl(\int_{0}^{\infty}\frac{(-1)^{n-1}s^{2n-2}e^{-\frac{s^2}{4}}}{2^{2n}(2n)!\pi}C(s)x\,ds\biggr)H_{2n}(t).$$
\end{itemize}

\end{theorem}

\begin{proof} (i) We write $\displaystyle{B(x):=\frac{1}{2^{2n} (2n)!}A^nT(\frac{1}{4})x=\int_{-\infty}^{\infty}h_{2n}(t)C(t)x\,dt,}$ for $x\in X,$ see Theorem \ref{loc2} (i). We apply Proposition \ref{propHerm} (iii) and integrate by parts to obtain  $$
B(x)
=\frac{1}{2^{2p}2n\ldots(2n-2p+1)}\int_{-\infty}^{\infty}h_{2n-2p}(t)A^pC(t)x\,dt,
\qquad x\in D(A^p).
$$
 By Theorem \ref{main2} (i), we get the following inequality  $$\lVert B(x) \rVert\leq \frac{1}{2^{2p}2n\ldots(2n-2p+1)}\lVert h_{2n-2p}\rVert_1\,\displaystyle\sup_{t\in\R}\lVert A^p C(t)x \rVert\leq\frac{C\sqrt{(2n-2p)!}}{2^{n+p}(2n)!}\lVert A^p x \rVert,$$ for $x\in D(A^p).$

\noindent (ii) Note that integrating by parts, and using Proposition \ref{propHerm} (iii) and above part (i), the inequality $$\lVert\displaystyle\int_{-\infty}^{\infty}h_{2n+1}(t)S(t)x\,dt\rVert= \frac{1}{2(2n+1)}\lVert\displaystyle\int_{-\infty}^{\infty}h_{2n}(t)AC(t)x\,dt\rVert\leq\frac{C\sqrt{(2n-2p)!}}{2^{n+p+1}(2n+1)!}\lVert A^{p+1} x \rVert$$ holds for $x\in D(A^{p+1}).$

\noindent (iii) Following the same steps as in Theorem \ref{rate1} and using part (i), we take $t\in\R,$  and $m_0\in\N$ such that for all $m\geq m_0$ with $1\leq p\leq m+1,$ \begin{eqnarray*}
\lVert C(t)x-C_{m}(t)x\rVert&\leq&\displaystyle\sum_{n=m+1}^{\infty}\lVert \int_{-\infty}^{\infty}h_{2n}(s)C(s)x\,ds \rVert |H_{2n}(t)| \cr &\leq&\displaystyle\sum_{n=m+1}^{\infty}\frac{C_t}{2^{p}n^{\frac{1}{12}}}\sqrt{\frac{(2n-2p)!}{2n!}}\lVert A^p x\rVert
\leq  \frac{C_{t,p}}{m^{p-\frac{11}{12}}}\lVert A^p x\rVert,
\end{eqnarray*} for $x\in D(A^p).$

Similarly using part (ii), we take $t\in\R,$  and $m_0\in\N$ such that for all $m\geq m_0$ with $1\leq p\leq m+1,$ $$
\lVert S(t)x-S_{m}(t)x\rVert\leq\displaystyle\sum_{n=m+1}^{\infty}\lVert \int_{-\infty}^{\infty}h_{2n+1}(s)S(s)x\,ds \rVert |H_{2n+1}(t)|\leq  \frac{C_{t,p}}{m^{p-\frac{5}{12}}}\lVert A^{p+1} x\rVert,
$$ for $x\in D(A^{p+1}).$

In both cases, the locally uniformly convergence in $t$ is shown as in Theorem \ref{rate1}.

\noindent (iv) The proof of this part is similar to the proof of Corollary \ref{corola}.
\end{proof}

\section{Examples, remarks and final comments}

\setcounter{theorem}{0}
\setcounter{equation}{0}
In this last section we present some concrete examples of $C_0$-groups and cosine functions and we apply our results to give their vector-valued Hermite expansions. We also comment some connections with  well-known (scalar) Hermite expansions for temperated distributions. Finally  we point out some open questions and interesting problems associated to other families of operators.

\subsection{Shift group and cosine function}
Let $L^p(\R)$ with $1\leq p<\infty.$ The shift group (or translation group) in $L^p(\R)$, $(T(t))_{t\in\R},$ defined by $$ T(t)f(x):=f(x-t)=\delta_t* f(x) \qquad x\in\R, $$ is an isometry $C_0$-group (we denote by $(\delta_t)_{t\in \R}$ the usual Dirac delta distribution concentrated at $t$). The infinitesimal generator $A$ is the usual derivation operator, $A=-\frac{d}{dx}.$ Furthermore, $$
T^{(g)}(z)f = g_{z}*f,$$ where $g_z$ is the Gaussian kernel, $g_z(r)=\displaystyle{{1\over \sqrt{4\pi z}}e^{-r^2/4z}},$ for $\Re z>0$ and $r\in \RR.$ By Theorem \ref{loc} (ii), we obtain the formula
$$
\delta_t\ast f=\displaystyle\sum_{n=0}^{\infty}(h_n\ast f)H_n(t)=\displaystyle\sum_{n=0}^{\infty}\frac{(-1)^n}{2^n n!}(g_{\frac{1}{4}}^{(n)}\ast f)H_n(t), \qquad f\in W^{(1),p}(\R)
$$
where $W^{(1),p}(\R)$ is the Sobolev space defined by $W^{(1),p}(\R)=\{f\in L^p(\R)\,\,|\,\,f'\in L^p(\R)\}$.

The distribution principal value of ${1\over x}$ is a temperated distribution whose is expanded in terms of Hermite polynomials, see for example  \cite[pp 193]{AMS}, \cite[Example 2.4]{CMO}. In general, given $T\equiv(T(t))_{t\in \R}$, a uniformly bounded $C_0$-group in a UMD space $X$, the Hilbert transform ${\mathcal H}^T$ associated to $T$ is defined by
$$
 {\mathcal H}^T(x)=\lim_{\varepsilon \to 0^+, N\to \infty}{i\over \pi}\int_{\varepsilon<\vert t\vert <N}{T(t)x\over t}\,dt, \qquad x\in X,
 $$
see \cite[section 5]{Mo}. If we consider the shift group on $L^p(\R;X)$ with $X$ a UMD Banach space and $p\in (1,\infty),$ $T(t)f(:=f(\cdot-t),$ then $$({\mathcal H}^T f)(s)=\lim_{\varepsilon \to 0^+, N\to \infty}{i\over \pi}\int_{\varepsilon<\vert t\vert <N}{f(s-t)\over t}\,dt, \qquad f\in L^p(\R;X),\, s\in\R,$$ is the classical Hilbert transform. Relations between Dirichlet and F\'{e}jer families and  the Hilbert transform ${\mathcal H}^T$ are studied in \cite{Fasangova}. It would be interested to obtain the Hermite expansion of the bounded operator ${\mathcal H}^T$.

Now we define also in $L^p(\R)$, the isometry cosine function $(C(t))_{t\in\R},$ by $$ C(t)f(x):=\frac{1}{2}(f(x-t)+f(x+t))=\frac{1}{2}(\delta_t*f(x)+\delta_{-t}*f(x)) \qquad x\in\R.$$  The  generator is $A=\frac{d^2}{dx^2}$ and is clear that
$$
C(t)f=\displaystyle\sum_{n=0}^{\infty}\frac{1}{2^{2n} (2n)!}(g_{\frac{1}{4}}^{(2n)}\ast f)H_{2n}(t), \qquad f\in W^{(2),p}(\R),
$$
where $W^{(2),p}(\R)$ is the Sobolev space defined by $W^{(2),p}(\R)=\{f\in L^p(\R)\,\,|\,\,f''\in L^p(\R)\}$.

\subsection{Multiplication groups and cosine functions}
Let $(\Omega, \Sigma,\mu)$ be a $\sigma$-finite measure space, $1\le p<\infty$ and the Lebesgue Banach space $L^p(\Omega)$. We consider the unitary $C_0$-group $(T(t))_{t\in\R}$ in $L^p(\Omega)$ defined by $$T(t)f(s)=e^{itq(s)}f(s),\qquad s\in \Omega,$$ where $q:\Omega\to \R$ is a Lebesgue measurable function. The infinitesimal  generator is  $A=iq,$ $D(A)=\{f\in L^p(\Omega)\,\,|\,\, qf\in L^p(\Omega)\}$. These multiplication groups are treated deeply in \cite[Chapter I, Section 4; Chapter II, Section 2.9]{Nagel}.
 Some known examples are the Fourier Transform of the Gaussian and Poisson semigroups, $q(s)=-s^2$ and $q(s)=-|s|;$
recently some interesting examples for $q(s)=-\text{log}(1+ s^2)$ and \mbox{$q(s)=-\text{log}(1+|s|),$} have been studied in \cite{Campos}.

 Note that $$T^{(g)}(\frac{1}{4})f(s)=\frac{1}{\sqrt{\pi}}\biggl(\int_{-\infty}^{\infty}e^{-x^2}e^{ixq(s)}\,dx\biggr) f(s)=e^{-\frac{q^2(s)}{4}}f(s),\qquad s\in \Omega.$$
By Theorem \ref{loc} (ii), we obtain that
$$
T(t)f(s)=\displaystyle\sum_{n=0}^{\infty}\frac{i^nq^n(s)}{2^n n!}e^{-\frac{q^2(s)}{4}}f(s)H_n(t), \qquad f\in D(A).
$$
For $q(s)=\lambda\in \R$, we get the formula (\ref{expo}).


 Now we suppose that  $m:\Omega \to\R^-$ is a Lebesgue measurable function. We consider the cosine function  $(C(t))_{t\in\R}$ in $L^p(\Omega)$ defined by $$C(t)f(s)=\cos(t\sqrt{-m(s)})f(s),\qquad s\in \Omega,$$ whose generator is $A=m,$ $D(A)=\{f\in L^p(\Omega)\,\,|\,\, mf\in L^p(\Omega)\}$. Note that $$T^{(c)}({1\over 4})f(s)=\frac{2}{\sqrt{\pi}}\biggl(\int_{0}^{\infty}e^{-t^2}\cos(t\sqrt{-m(s)})\,dt\biggr) f(s)=e^{\frac{m(s)}{4}}f(s),\qquad s\in \Omega,$$ where he have used \cite[Formula 7.80, p.66]{Badii}. By Theorem \ref{loc2} (ii), we have that
$$
C(t)f(s)=\displaystyle\sum_{n=0}^{\infty}\frac{m^n(s)}{2^{2n} (2n)!}e^{\frac{m(s)}{4}}f(s)H_{2n}(t), \qquad f\in D(A),
$$
for $t\in \R$ and $s\in \Omega$; in the particular case $m(s)=-a$, (with $a>0$) we obtain the formula (\ref{coseno}).

\subsection{Cosine functions on sequence spaces}
Let $X=c_0, \ell^p$ be spaces of all complex sequence $x=(x_k)_{k\in\N}$ convergent to $0$  equipped with the usual  norm $\lVert x\rVert_{\infty}:=\displaystyle\max_{k\in\N}|x_k|$ for $X=c_0$; and $\Vert x\Vert_p<\infty$ where
$\displaystyle{
\Vert x\Vert_p:=\left(\sum_{k=1}^\infty\vert x_k\vert^p\right)^{1\over p}}
$
 for $X=\ell^p$ with $1\le p <\infty$.

 For each $n\in\N$ let $e_n$ be the element of $X$ such that $(e_n)_k=\delta_{n,k}$ the Kronecker delta. Every $x\in X$ can be represented as the series $x=\displaystyle{\sum_{k=1}^{\infty}x_k e_k.}$
  The family of linear operators $(C(t))_{t\in\R}$ given by $$C(t)x=\displaystyle\sum_{k=1}^{\infty}\cos(kt)x_k e_k,\qquad x\in X,$$ is a strongly continuous cosine family of contractions on $X,$ generated by $(A, D(A))$ where
  $$Ax=\displaystyle\sum_{k=1}^{\infty}-k^2 x_k e_k, \qquad x\in D(A),$$ and $D(A)=\{ x\in X\,|\, (k^2 x_k)_{k=1}^\infty\in X \}.$ For $X=c_0$, this example has been studied in \cite[Example 2]{Bobrowski} in an approximation process.

  Note that $\displaystyle{T^{(c)}(\frac{1}{4})x}=\displaystyle\sum_{k=1}^{\infty}e^{-\frac{k^2}{4}}x_k e_k,$ and
  $$
  A^nT^{(c)}(\frac{1}{4})x= (-1)^n\sum_{k=1}^\infty k^{2n}e^{-\frac{k^2}{4}}x_k e_k, \qquad x\in X.
  $$
In this case, for $1\le p <  \infty$, we have that
$$
\Vert  A^nT^{(c)}(\frac{1}{4})\Vert_{{\mathcal B}(X)}= \sup_{k\ge 1}(k^{2n}e^{-\frac{k^2}{4}})\cong \left({4n\over e}\right)^n\cong C \frac{2^n \sqrt{(2n)!}}{n^{\frac{1}{4}}}, \qquad n \ge 1.
$$
Note that  Theorem \ref{loc2} (i) doesn't provide the optimal bound in this case. Then by Theorem 4.1 (ii), we get \begin{eqnarray*}
C(t)x&=&\displaystyle\sum_{n=0}^{\infty}\frac{(-1)^n}{2^{2n}(2n)!}\biggl(\displaystyle\sum_{k=1}^{\infty}k^{2n}e^{-\frac{k^2}{4}}x_k e_k\biggr) H_{2n}(t), \qquad x\in D(A),\quad t\in \RR,
\end{eqnarray*}
which is a vector-valued version of the identity (\ref{coseno}).

\subsection{Matrix approach to  cosine functions and $C_0$-groups}\label{sub}

(\cite[Theorem 3.14.11]{ABHN}) Let $A$ be an operator on a Banach space $X.$ Then the following assertions are equivalent,  \begin{itemize}
\item[(i)] $A$ generates a cosine function $(C(t))_{t\in\R}$ on $X.$
\item[(ii)] There exists a Banach space $V$ such that $D(A)\hookrightarrow V\hookrightarrow X$ and such that the operator $\mathcal{B},$ given by \begin{displaymath}
\mathcal{B}\left( \begin{array}{c}
x\\
y
\end{array}\right):=\left( \begin{array}{cc}
0 & I\\
A & 0
\end{array}\right)\left( \begin{array}{c}
x\\
y
\end{array}\right)=\left( \begin{array}{c}
y\\
Ax
\end{array}\right),
\end{displaymath}
with $D(\mathcal{B}):=D(A)\times V,$ generates a $C_0$-group $(\mathcal{J}(t))_{t\in \R}$ on $V\times X$  with the norm $\lVert (x,y)\rVert_{V\times X}=\lVert x\rVert_V+\lVert y\rVert_X$.  In this case \begin{displaymath}
\mathcal{J}(t)=\left(\begin{array}{cc}
C(t) & S(t) \\
AS(t) & C(t)
\end{array}\right), \qquad t\in \RR,
\end{displaymath} where $S(t)x=\int_0^t C(s)x\,ds,$ for $x\in X$ and $t\in \RR$.
\end{itemize}
Applying Theorem 3.1 (ii), we  expand $C_0$-group $(\mathcal{J}(t))_{t\in \R}$ through Hermite polynomials: for $(x,y)\in D(\mathcal{B})$ $$\mathcal{J}(t)(x,y)=\displaystyle\sum_{n=0}^{\infty}\mathcal{C}_n(x,y) H_n(t),\qquad t\in \R,$$ where $\mathcal{C}_n=\displaystyle{\frac{1}{2^n n!}\mathcal{B}^n\mathcal{J}^{(g)}\biggl(\displaystyle{\frac{1}{4}}\biggr)}$ for $n\ge 0$. Note that \begin{displaymath}
\mathcal{B}^{2n}=\left( \begin{array}{cc}
A^n & 0\\
0 & A^n
\end{array}\right),\ \mathcal{B}^{2n+1}=\left( \begin{array}{cc}
0 & A^n\\
A^{n+1} & 0
\end{array}\right),\qquad n\geq0,
\end{displaymath}
and
\begin{displaymath}
\mathcal{J}^{(g)}(\frac{1}{4})(x,y)=\frac{1}{\sqrt{\pi}}\int_{-\infty}^{\infty}e^{-t^2}\mathcal{J}(t)(x,y)\,dt=\left( \begin{array}{cc}
T^{(c)}(\frac{1}{4}) & 0\\
0 & T^{(c)}(\frac{1}{4})
\end{array}\right)\left( \begin{array}{c}
x\\
y
\end{array}\right),\qquad (x,y)\in V\times X,
\end{displaymath}
where we have used that $$T^{(c)}(\frac{1}{4})=\frac{1}{\sqrt{\pi}}\int_{-\infty}^{\infty}e^{-t^2}C(t)\,dt, \qquad \frac{1}{\sqrt{\pi}}\int_{-\infty}^{\infty}e^{-t^2}AS(t)\,dt=0.$$ Then \begin{displaymath}
\mathcal{C}_{2n}=\frac{1}{2^{2n} (2n)!}\left( \begin{array}{cc}
A^nT^{(c)}(\frac{1}{4}) & 0\\
0 & A^nT^{(c)}(\frac{1}{4})
\end{array}\right),
\quad \mathcal{C}_{2n+1}=\frac{1}{2^{2n+1} (2n+1)!}\left( \begin{array}{cc}
0 & A^nT^{(c)}(\frac{1}{4})\\
A^{n+1}T^{(c)}(\frac{1}{4}) & 0
\end{array}\right),
\end{displaymath}
for $n\geq0.$
In other hand, we apply Theorem \ref{loc2}(ii) to get that
\begin{eqnarray*}
\mathcal{J}(t)(x,y)&=&\left(\begin{array}{cc}
C(t) & S(t) \\
AS(t) & C(t)
\end{array}\right)\left( \begin{array}{c}
x\\
y
\end{array}\right)\cr&=& \sum_{n=0}^\infty \frac{1}{2^{2n} (2n)!}\left( \begin{array}{cc}
A^nT^{(c)}(\frac{1}{4}) & 0\\
0 & A^nT^{(c)}(\frac{1}{4})
\end{array}\right)\left( \begin{array}{c}
x\\
y
\end{array}\right)H_{2n}(t)\cr &\quad&\qquad \qquad+ \sum_{n=0}^\infty\frac{1}{2^{2n+1} (2n+1)!}\left( \begin{array}{cc}
0 & A^nT^{(c)}(\frac{1}{4})\\
A^{n+1}T^{(c)}(\frac{1}{4}) & 0
\end{array}\right)\left( \begin{array}{c}
x\\
y
\end{array}\right)H_{2n+1}(t)\cr
&=&\sum_{n=0}^\infty\mathcal{C}_{n}(x,y)H_n(t)\cr
\end{eqnarray*}
for $(x,y)\in D(\mathcal{B})$ and $t\in \R$.

In the case that we consider the Banach space $X\times X$ with the norm $\lVert (x,y)\rVert_{X\times X}=\lVert x\rVert+\lVert y\rVert$, the operator $A$ generates a cosine function if and only if $(\mathcal{A}, D(\mathcal{A}))$ generates a once integrated semigroup $\mathcal{S}$ on $X\times X,$ where $D(\mathcal{A})=D(A)\times X$,
\begin{displaymath}
\mathcal{A}:=\left( \begin{array}{cc}
0 & I\\
A & 0
\end{array}\right), \qquad \mathcal{S}(t)=\left(\begin{array}{cc}
S(t) & \int_0^t S(s)\,ds\\
C(t)-I & S(t)
\end{array}\right), \quad t\in \R,
\end{displaymath}
see \cite[Theorem 3.14.7]{ABHN} and \cite[Definition 3.2.1]{ABHN}. Hermite expansions for $n$-times integrated groups have not considered in the literature and seems to be natural to develop this theory. In \cite[Theorem 4.1]{jara},  stable rational approximations  for exponential functions are considered to approximate $n$-times integrated semigroups (and then  cosine functions) for smooth initial data. Both approaches might be compared in a forthcoming paper.

\subsection*{Acknowledgements} Authors thank O. Ciaurri and L. Roncal some advices and comments provided for obtaining some results. We also thank an anonymous referee for several comments, ideas and references which have helped to improve this version of the paper.

\end{document}